\documentclass[11pt]{article}
\usepackage{bbm}
\usepackage{geometry}
\usepackage{color}
\usepackage{amsfonts}
\usepackage{latexsym, amssymb, amsmath, amscd, amsthm, amsxtra}
\usepackage{mathtools}
\usepackage{enumerate}
\usepackage[all]{xy}
\usepackage{mathrsfs}
\usepackage{fancyhdr}
\usepackage{listings}
\usepackage{hyperref}

\def\P{\mathbb{P}}

\def\E{\mathbb{E}}

\def\11{\mathbbm{1}}

\newcommand\dif{\mathop{}\!\mathrm{d}}

\newtheorem{thm}{Theorem}[section]

\newtheorem{proposition}[thm]{Proposition}

\newtheorem{lemma}[thm]{Lemma}

\theoremstyle{definition}
\newtheorem{remark}[thm]{Remark}

\numberwithin{equation}{section}

\begin{document}
\title{Detection threshold for correlated Erd\H{o}s-R\'enyi graphs via densest subgraphs}

\author{Jian Ding\\Peking University \and Hang Du\\Peking University}

\maketitle

\begin{abstract}
	The problem of detecting edge correlation between two Erd\H{o}s-R\'enyi random graphs on $n$ unlabeled nodes can be formulated as a hypothesis testing problem: under the null hypothesis, the two graphs are sampled independently; under the alternative, the two graphs are independently sub-sampled from a parent graph which is Erd\H{o}s-R\'enyi $\mathbf{G}(n, p)$ (so that their marginal distributions are the same as the null). We establish a sharp information-theoretic threshold when $p = n^{-\alpha+o(1)}$ for $\alpha\in  (0, 1]$ which sharpens a constant factor in a recent work by Wu, Xu and Yu. A key novelty in our work is an interesting connection between the detection problem and the densest subgraph of an Erd\H{o}s-R\'enyi graph.
\end{abstract}

\maketitle


\section{Introduction}\label{sec:intro}

In this paper, we study the information-theoretic threshold for detecting the correlation between a pair of Erd\H{o}s-R\'enyi graphs. To put this question into a precise mathematical framework, we first need to choose a probabilistic model for a pair of correlated Erd\H{o}s-R\'enyi graphs, and one natural choice is to obtain the two graphs as two independent subsamplings from a common Erd\H{o}s-R\'enyi graph. More formally, for two vertex sets $V, \mathsf V$ of cardinality $n$ we let $E_0$ be the set of unordered pairs $(u,v)$ with $u,v\in V$ and $u\neq v$, and let $\mathsf E_0$ be the set of unordered pairs $(\mathsf u, \mathsf v)$ with $\mathsf u, \mathsf v\in \mathsf V$ and $\mathsf u\neq \mathsf v$. For some model parameters $p,s\in (0,1)$ (which may depend on $n$), we sample a uniform bijection $\pi^*$ between $V$ and $\mathsf V$, independent Bernoulli variables $\{J_{u, v}: (u,v)\in E_0\}$ with parameter $p$ and independent Bernoulli variables $\{I_{u, v}: (u,v)\in E_0\}, \{\mathsf I_{\mathsf u, \mathsf v}: (\mathsf u, \mathsf v)\in \mathsf E_0\}$ with parameter $s$, and then we define
\begin{equation}\label{eq:rule}
	(G_{u,v},\mathsf G_{\pi^*(u),\pi^*(v)})=(J_{u,v}I_{u,v},J_{u,v}\mathsf I_{\pi^*(u),\pi^*(v)})\,.
\end{equation}
Then $(G, \mathsf G)$ forms a pair of correlated Erd\H{o}s-R\'enyi graphs where the edge set of $E$ (respectively $\mathsf E$) consists of all $(u, v)\in E_0$ (respectively $(\mathsf u, \mathsf v)\in \mathsf E_0$) such that $G_{u, v}=1$ (respectively $\mathsf G_{\mathsf u, \mathsf v} =1$). Note that marginally each $G$ and $\mathsf G$ is an Erd\H{o}s-R\'enyi graph on $n$ vertices with edge probability $ps$, whose law we denote as $\mathbf{G}(n, ps)$.

Therefore, one (natural) version for the problem of detecting correlated Erd\H{o}s-R\'enyi graphs can be formulated as a hypothesis testing problem, where the null hypothesis $\operatorname{H}_0$ and the alternative hypothesis $\operatorname{H}_1$ are given as 
\begin{align*}
	\operatorname{H}_0: &\ (G,\mathsf G)\sim \mathsf P,\text{ which is a pair of independent Erd\H{o}s-R\'enyi Graphs }\mathbf{G}\left(n,ps\right),\\
	\operatorname{H}_1: &\ (G,\mathsf G)\sim \mathsf Q,\text{ which is a pair of correlated graphs given by the rule \eqref{eq:rule}} \,.
\end{align*}
Our goal is to test $\operatorname{H}_0$ versus $\operatorname{H}_1$ given $(G,\mathsf G)$ as observations while $\pi^*$ remains to be unknown. It is well-known that the testing error is captured by the total variation distance between the null and alternative distributions, and our main contribution is to establish a sharp phase transition on this total variation distance in the sparse regime.
\begin{thm}\label{thm1}
	Suppose $p=p(n)$ satisfies $p=n^{-\alpha+o(1)}$ for some $\alpha\in(0,1]$ as $n\to \infty$. Let $\lambda_*=\varrho^{-1}(\frac{1}{\alpha})$ (where $\rho$ is defined in \eqref{eq-densest-subgraph} below), then for any constant $\varepsilon>0$, the following holds.
	For $\alpha\in (0,1]$, if $s$ satisfies $nps^2\ge \lambda_*+\varepsilon$, then
	\begin{equation}\label{eq:positive}
		\operatorname{TV}(\mathsf{P},\mathsf{Q})=1-o(1) \mbox{ as }n\to\infty,
	\end{equation}
	where $\operatorname{TV}(\mathsf{P},\mathsf{Q}) = \frac{1}{2}\sum_{\omega} |\mathsf P[\omega] - \mathsf Q[\omega]|$ is the total variation distance between $\mathsf P$ and $\mathsf Q$.
	In addition, for $\alpha\in (0,1)$, if $s$ satisfies $nps^2\le \lambda_*-\varepsilon$, then
	\begin{equation}\label{eq:negative-<1}
		\operatorname{TV}(\mathsf{P},\mathsf{Q})=o(1) \mbox{ as }n\to\infty.
	\end{equation} 
\end{thm}
Our work is closely related to and much inspired by a recent work \cite{WXY20+}, where a sharp threshold was established for $\alpha = 0$ and upper and lower bounds on $\lambda_*$ up to a constant factor were established for $\alpha\in (0, 1]$. In particular, Theorem~\ref{thm1} solves \cite[Section 6, Open Problem 3]{WXY20+}. It is worth emphasizing that \eqref{eq:negative-<1} for $\alpha = 1$ with $np\to \infty$ was already proved in \cite{WXY20+} (which even allows $\epsilon \to 0$ as long as $\epsilon \gg n^{-1/3}$). While our method should also be able to give \eqref{eq:negative-<1} for $\alpha = 1$, we chose to exclude this case since the assumption $\alpha<1$ allows to avoid some technical complications. Thus, the only remaining case is when $np$ has order 1, in which there is no sharp phase transition.  Indeed, on the one hand,  for some constant $c>0$ satisfying $np\ge c$ and $s\ge c$ we have 
\begin{equation}\label{eq-lower-bound-s-positive}
	\liminf_{n\to\infty}\operatorname{TV}(\mathsf P,\mathsf Q)>0\,.
\end{equation} 
This follows readily by comparing the marginal distribution of the pair $$\left(\frac{|E|-\binom{n}{2}ps}{\sqrt{\binom{n}{2}ps(1-ps)}},\frac{|\mathsf E|-\binom{n}{2}ps}{\sqrt{\binom{n}{2}ps(1-ps)}}\right)$$ under $\mathsf P$ and $\mathsf Q$. A straightforward application of Central Limit Theorem yields that the marginal law of such pair is approximately a pair of independent normal variables with mean zero and variance 1 under $\mathsf P$, and in contrast is a pair of bivariate normal variables with mean zero, variance 1 and with correlation at least $c$ under $\mathsf Q$. On the other hand, as shown in \cite{WXY20+}, if $s\leq 0.1$ and $ n^{1/3}(1-nps^2 )\to \infty$ then $\limsup_{n\to\infty}\operatorname{TV}(\mathsf P,\mathsf Q) < 1$.

\noindent {\bf Background and related results.} Recently, there has been extensive study on the problem of detecting correlation between two random graphs and the closely related problem of matching the vertex correspondence in the presence of correlation. Questions of this type have been raised from various applied fields such as social network analysis \cite{NS08,NS09}, computer vision \cite{CSS07,BBM05}, computational biology \cite{SXB08,VCP15} and natural language processing \cite{HNM05}. 

Despite the fact that Erd\H{o}s-R\'enyi Graph perhaps does not quite capture important features for any network arising from realistic problems, (similar to most problems on networks) it is plausible that a complete understanding for the case of Erd\H{o}s-R\'enyi Graphs forms an important and necessary step toward the much more ambitious goal of mathematically understanding graph detection and matching problems for realistic networks arising from applications (note that for many applications it remains a substantial challenge to propose a reasonable underlying random graph model). Along this line, many progress has been made recently, including information-theoretic analysis \cite{CK16, CK17, HM20, WXY20+,WXY21+} and proposals for various efficient algorithms  \cite{PG11, YG13, LFP14, KHG15, FQRM+16, SGE17, BCL19, DMWX21, BSH19, CKMP19,DCKG19, MX20, FMWX20,  GM20, FMWX19+, MRT21+, MWXY21+}.  Out of these references, the ones closely related to our work include (the aforementioned) \cite{WXY20+} and \cite{WXY21+} which studied the information-theoretic threshold for the matching problem, as well as \cite{MWXY21+} which obtained an efficient algorithm for detection when the correlation between the two graphs is above a certain constant. As of now, a huge information-computation gap remains for both detection and matching problems, and it is a major challenge to completely understand the phase transition for the computational complexity for either detection or matching problems.

Recently, detection and matching problems have also been studied for models other than  Erd\H{o}s-R\'enyi. For instance, a model for correlated randomly growing graphs was studied in \cite{RS20+}, graph matching for correlated stochastic block model was studied in \cite{RS21} and graph matching for correlated random geometric graphs was studied in \cite{WWXY22+}. A very interesting direction is to design efficient algorithms for graph detection and matching that is robust to the underlying random graph models.

\noindent {\bf A connection to densest subgraph.} In \cite{WXY20+}, the authors used the maximal overlap between the two graphs over all vertex bijections as the testing statistic. This is a natural and likely efficient statistic, although it is not so easy to analyze the maximal overlap in the correlated case so \cite{WXY20+} lower-bounded it by the overlap given by the true matching. While this relaxation manages to capture the detection threshold in the dense regime (when $p = n^{o(1)}$), it only captures the threshold up to a constant factor in the sparse regime (when $p = n^{- \alpha+o(1)}$ for $\alpha>0$). Here is a brief description on the insights behind that guided this paper: In the dense regime, near the threshold the intersection of two correlated random graphs is an  Erd\H{o}s-R\'enyi with large degree and thus it is more ``regular'' in a sense that the densest subgraph is more or less as dense as the whole graph. In the sparse regime, on the one hand, the intersection is an  Erd\H{o}s-R\'enyi with constant degree and in this case the spatial fluctuation plays a non-negligible role such that the densest subgraph has significantly higher average degree than the whole graph; on the other hand, the maximal intersection of two independent random graphs (over all vertex bijections) is much more ``regular'' than an  Erd\H{o}s-R\'enyi such that its densest subgraph has about the same average degree as the whole graph (as proved in \eqref{eq-densest-subgraph-intersection-independent}). In summary, this suggests that a more efficient testing statistic is the maximal densest subgraph over all vertex bijections as defined in \eqref{eq:test-statistic}, which indeed yields the correct upper bound on the detection threshold (see Theorem~\ref{algo}). The major technical contribution of this paper is then to prove the sharp lower bound on the detection threshold (i.e., as in \eqref{eq:negative-<1}). To this end, we use a similar truncated second moment method as employed in \cite{WXY20+} with the additional insight that the truncation should be related to the densest subgraph. We will discuss more on this in Section~\ref{sec:impossible}.

Next we briefly describe the development on the densest subgraph for an Erd\H{o}s-R\'enyi graph. This problem arose in the study of load balancing problem \cite{Hajek90}, a particular example of which is to balance the loads when assigning  $m$ balls into $n$ bins subject to the constraint that each ball is assigned to either of two randomly chosen bins. The load balancing problem is also closely related to the emergence of a $k$-core in an Erd\H{o}s-R\'enyi graph (a $k$-core is a maximal connected subgraph  in which all vertices have degree at least $k$), and much progress has been made in \cite{CSW07, FR07, GW10, FKP16}. While \cite{CSW07, FR07} made significant progress in understanding the densest subgraph (note that \cite{GW10, FKP16} are in the context of hypergraphs), the asymptotic behavior for the maximal subgraph density of an Erd\H{o}s-R\'enyi graph (with average degree of order $1$) was only established in \cite{AS16}. In fact, the authors of \cite{AS16} managed to compute the asymptotic value for the maximal subgraph density of a random graph with prescribed degree sequence using the objective method from \cite{AS04}, and their result on Erd\H{o}s-R\'enyi graphs (see Proposition~\ref{prop:densest-subgraph}) is crucial for our work. 

\noindent\textbf{Acknowledgements. }We warmly thank Nicholas Wormald, Yihong Wu and Jiaming Xu for stimulating discussions. Hang Du is partially supported by the elite undergraduate training program of School of Mathematical Science in Peking University.

\medskip

\section{Detect correlation via densest subgraph}\label{sec:detection}

\subsection{The densest subgraph of an  Erd\H{o}s-R\'enyi graph}\label{sec:densest-subgraph}

The following result of \cite{AS16} provides an important input for the proof of Theorem~\ref{thm1}.

\begin{proposition}[\cite{AS16}, Theorem 1, Theorem 3]\label{prop:densest-subgraph}
	For any constant $\lambda>0$, there exists a constant $\varrho(\lambda)>0$ which can be explicitly written via a variational characterization, such that for an Erd\H{o}s-R\'enyi graph $\mathcal{H}(V,\mathcal{E})\sim \mathbf{G}(n,\frac{\lambda}{n})$, 
	\begin{equation}\label{eq-densest-subgraph}
		\max_{\emptyset\neq U\subset V} \frac{| \mathcal{E}(U)|}{|U|} \to \varrho(\lambda) \mbox{ in probability as } n\to\infty,
	\end{equation}
	where $ \mathcal E(U)$ is the collection of edges in $\mathcal E$ with both endpoints in $U$. Further, the function $\varrho(\cdot)$ satisfies 	\begin{equation}
		1\le \frac{\varrho(\beta)}{\varrho(\alpha)}\le \frac{\beta}{\alpha}, \ \forall\ 0<\alpha<\beta,
	\end{equation}
	hence $\varrho(\cdot)$ is continuous and increasing.
\end{proposition}
\begin{remark}
	Although \cite{AS16} treated Erd\H{o}s-R\'enyi graphs with $n$ vertices and $\lfloor\lambda n\rfloor$ edges for a fixed constant $\lambda>0$, counterparts of all results apply to the $\mathbf G(n, \frac{\lambda}{n})$ model since the total number of edges in $\mathbf G(n, \frac{\lambda}{n})$ concentrate around $\frac{\lambda n}{2}$. In addition, our definition of $\varrho$ is different from that in \cite{AS16} by a scaling factor $2$, i.e. $\varrho(\lambda)$  for us equals to $\varrho(\frac{\lambda}{2})$ for $\varrho$ in \cite{AS16}.
\end{remark}
We call the largest subgraph that maximizes the left hand side of \eqref{eq-densest-subgraph} (if there are many such subgraphs, pick one of them arbitrarily) as \emph{the densest subgraph}. In order to establish a sharp threshold phenomenon for graph detection, we need $\rho$ to be a \emph{strictly} increasing function, as proved in the following proposition. 
\begin{proposition}\label{prop:densest-subgraph-is-linear}
	For $\lambda>1$, $\varrho(\lambda)>1$ and $\varrho$ is strictly increasing. Furthermore, there exists some constant $c_\lambda>0$, such that  with probability tending to $1$ as $n\to \infty$, the size of the densest subgraph in an Erd\H{o}s-R\'enyi graph $\mathcal H\sim\mathbf{G}(n,\frac{\lambda}{n})$ is at least $c_\lambda n$. 
\end{proposition}
It is readily to see that $\varrho(\lambda)=1$ for $\lambda\le 1$ (For $\lambda<1$ this follows from the well-known fact \cite{ER60} that with probability tending to 1, an Erd\H{o}s-R\'enyi graph in the sub-critical phase has a component of size of order $\log n$ and all components have at most one cycle; for $\lambda=1$ this follows from continuity). By Proposition~\ref{prop:densest-subgraph-is-linear}, we can define the inverse function $\varrho^{-1}:[1,\infty)\to [1,\infty)$, where we let $\varrho^{-1}(1)=1$. In the proof of Proposition~\ref{prop:densest-subgraph-is-linear} and some estimates later, we need the following Chernoff bound for Bernoulli variables (see \cite[Theorem 4.4]{MU05}): For $X\sim \operatorname{Bin}(n,p)$, denote $\mu=np$, then for any $\delta>0$, 
\begin{equation}\label{eq:chernoff}
	\P[X\ge (1+\delta)\mu]\le \exp\left(-\mu[(1+\delta)\log(1+\delta)-\delta]\right).
\end{equation} 
\begin{proof}[Proof of Proposition~\ref{prop:densest-subgraph-is-linear}]
	First we show $\varrho(\lambda)>1$ for any $\lambda>1$. This follows immediately from the fact that with probability tending to 1, the 2-core in $\mathcal H$ contains $(1-x)(1-\frac{x}{\lambda}+o(1))n$ vertices and $(1-\frac{x}{\lambda}+o(1))^2\frac{\lambda n}{2}$ edges, where $x\in (0,1)$ is given by the equation $xe^{-x}=\lambda e^{-\lambda}$ (See \cite[Theorem 3]{Pittel90} and see also e.g. \cite[Lemma 2.16]{AM22}).
	
	Next we show that with probability tending to 1, the number of vertices in the densest subgraph has at least $c_\lambda n$ vertices for some constant $c_\lambda>0$. Since $\varrho>1$, there exists a $\rho=\rho(\lambda)>1$ such that with probability tending to 1 the densest subgraph has edge-vertex ratio greater than $\rho$. Let $\mathcal{A}$
	be the event that there exists a subgraph of $\mathcal H$ which has at most $c_\lambda n$ vertices and has edge-vertex ratio greater than $\rho$. Then, by a simply union bound,
	\begin{align*}
		\P[\mathcal{A}]\le&\sum_{k\leq c_\lambda n}\binom{n}{k}\P\left[\operatorname{Bin}\left(\frac{k(k-1)}{2},\frac{\lambda}{n}\right)\ge \rho k\right]\\
		\stackrel{\eqref{eq:chernoff}}{\le}&\sum_{k\le c_\lambda n}\frac{n^k}{k!}\exp\left(-\rho k\log\left(\frac{2n}{\lambda(k-1)}\right)+\rho k\right)\\
		\leq& \sum_{k\le c_\lambda n}\exp\left(-(\rho-1)k\log \left(\frac{n}{k}\right)+Ck\right),
	\end{align*}
	where $\operatorname{Bin}(\frac{k(k-1)}{2},\frac{\lambda}{n})$ (similar notation applies below later) denotes a binomial variable (which is the distribution for the number of edges in a subgraph with $k$ vertices), and $C=C(\lambda,\rho)$ is a large constant. Once we choose $c_\lambda>0$ small enough such that $(\rho-1)\log c_\lambda^{-1}>C$, then $\P[\mathcal{A}]=o(1)$, yielding the desired result. 
	
	Now we are ready to show $\varrho$ is strictly increasing by a simple sprinkling argument. For any $1<\lambda_1<\lambda_2$, we have shown that the densest subgraph $A$ in $\mathcal H_{\lambda_1}\sim \mathcal{G}\left(n,\frac{\lambda_1}{n}\right)$  has at least $c_{\lambda_1} n$ vertices with probability tending to 1. We now independently sample another Erd\H{o}s-R\'enyi graph $\mathcal H_{\lambda_2-\lambda_1}\sim \mathcal{G}\left(n,\frac{\lambda_2-\lambda_1}{n}\right)$ in the same vertex set as for $\mathcal H_{\lambda_1}$, then with probability tending to 1, $\mathcal H_{\lambda_2-\lambda_1}$ contains more than $\frac{(\lambda_2-\lambda_1)c_{\lambda_1}^2}{4}n$ edges within $A$. Since $\mathcal H_{\lambda_2}$ stochastically dominates $\mathcal H_{\lambda_1}\cup \mathcal H_{\lambda_2-\lambda_1}$, this gives $\varrho(\lambda_2)>\varrho(\lambda_1)$.
\end{proof}

\subsection{Test graph correlation}

We next define our testing statistic. For any bijection $\pi: V \to \mathsf V$, we define the $\pi$-intersection graph $ {\mathcal H}_\pi$ of $G$ and $\mathsf G$ as
\begin{equation}\label{def:pi-intersection-graph}
	{\mathcal H}_\pi = (V, {\mathcal E}_\pi),\text{ where }(u,v)\in \mathcal{E}_\pi \mbox{ if and only if } G_{u, v} =  \mathsf G_{\pi(u), \pi(v)} = 1.
\end{equation}
Then our testing statistic is defined by
\begin{equation}\label{eq:test-statistic}
	\mathcal{T}(G,\mathsf G)\stackrel{\operatorname{\Delta}}{=}\max_\pi\max_{U\subset V:|U|\ge n/\log n}\frac{|{\mathcal{E}}_\pi(U)|}{|U|}.
\end{equation}
\begin{thm}\label{algo}
	Suppose $p=p(n)$ satisfies $p=n^{-\alpha+o(1)}$ for some $\alpha\in(0,1]$ as $n\to \infty$. Let $\lambda_*=\varrho^{-1}(\frac{1}{\alpha})$. For $s$ satisfying $nps^2\ge \lambda_*+\varepsilon$, let $\tau= \frac{\varrho(\lambda_*)+\varrho(\lambda_*+\varepsilon)}{2}$ and let $\mathcal{T}(G,\mathsf G)$ be defined as in \eqref{eq:test-statistic}. Then $\mathcal{T}(G,\mathsf G)$ with threshold $\tau$ achieves strong detection, i.e.
	\begin{equation}
		\mathsf Q[\mathcal{T}(G,\mathsf G)<\tau] + \mathsf P[\mathcal{T}(G,\mathsf G)\geq \tau] = o(1)\,.
	\end{equation}
\end{thm}
Note that Theorem~\ref{algo} implies \eqref{eq:positive}.
\begin{proof}[Proof of Theorem~\ref{algo}]
	Denote $\lambda=nps^2$. By Proposition~\ref{prop:densest-subgraph-is-linear}, 
	\begin{equation}\label{eq-tau-rho-alpha}
		\frac{1}{\alpha}=\varrho(\lambda_*)<\tau<\varrho(\lambda_*+\varepsilon).
	\end{equation}	
	Under $\operatorname{H}_1$, we see that ${\mathcal{H}}_{\pi^*}$ is an Erd\H{o}s-R\'enyi graph $\mathcal{G}\left(n,\frac{\lambda}{n}\right)$. Since $\tau<\varrho(\lambda_*+\varepsilon)$,  by Propositions~\ref{prop:densest-subgraph} and \ref{prop:densest-subgraph-is-linear}, with probability $1-o(1)$ the densest subgraph of $\mathcal H_{\pi^*}$ has edge-vertex ratio at least $\tau$ and has at least $c_\lambda n\ge n/\log n$ many vertices. This implies that $\mathsf Q[\mathcal{T}(G,\mathsf G)<\tau]=o(1)$.
	
	Furthermore, by a union bound, we get that
	$$\mathsf{P}[\mathcal{T}(G,\mathsf G)\ge \tau]\le \sum_{k\ge n/\log n}\binom{n}{k}^2 k!\P\left[\operatorname{Bin}\left(\frac{k(k-1)}{2},(ps)^2\right)\ge \tau k\right], $$
	where $\binom{n}{k}^2$ accounts for the number of ways to choose $k$ vertices $A\subset V$ and $k$ vertices $\mathsf A\subset \mathsf V$, $k!$ accounts for the number of bijections between $A$ and $\mathsf A$. Crucially, we only need to take a union bound over $k!$ bijections since the subgraph of $\mathcal H_\pi$ on $A$, subject to the constraint that $\pi$ maps $A$ to $\mathsf A$, only depends on $\pi|_A$.
	Since $p=n^{-\alpha+o(1)}$, we have $\log(\frac{n}{kp})\ge[\alpha+o(1)]\log n$. In addition,
	$\tau>\frac{1}{\alpha}$ as in \eqref{eq-tau-rho-alpha}. Then, we get that for a small positive constant $\delta=\delta(\alpha,\tau)$,
	\begin{align*}	
		\mathsf{P}[\mathcal{T}(G,\mathsf G)\ge \tau]\stackrel{\eqref{eq:chernoff}}{\le} &\sum_{k\ge n/\log n}\frac{n^{2k}}{k!}\exp\left(-\tau k\log\left(\frac{2\tau n}{\lambda kp}\right)+\tau k\right)\\
		=&\sum_{k\ge n/\log n}\frac{n^{2k}}{k!}\exp\left(-(1+2\delta )k\log n+o(k\log n)\right).
	\end{align*}
	By a straightforward computation, we then get that when $n$ is large enough, 
	\begin{equation*}
		\mathsf{P}[\mathcal{T}(G,\mathsf G)\ge \tau]	\le \sum_{k\ge n/\log n}\frac{n^{(1-\delta)k}}{k!}\le \frac{n^{(1-\delta)\lfloor n/\log n\rfloor}}{\lfloor n/\log n\rfloor!}\sum_{t\ge 0}\left(\frac{n^{1-\delta}}{n/\log n}\right)^t 
		=\ o(1). \label{eq-densest-subgraph-intersection-independent}
	\end{equation*}
	This proves $\mathsf Q[\mathcal T(G,\mathsf G)< \tau]+\mathsf P[\mathcal{T}(G,\mathsf G)\geq \tau] = o(1)$ as required.
\end{proof}

\section{Impossibility for detection}\label{sec:impossible}

This section is devoted to the proof of \eqref{eq:negative-<1}.
The basic idea follows the framework of conditional second moment for the likelihood ratio as in \cite{WXY21+}, with the aforementioned additional key insight that the truncation should involve the densest subgraph. It turns out somewhat more convenient in our case to compute the conditional first moment for $\frac{\dif \mathsf Q}{\dif\mathsf P}$ when $(G, \mathsf G)\sim \mathsf Q$, which is equivalent to its second moment when $(G, \mathsf G)\sim \mathsf P$. We choose the first moment formulation since it is then convenient to consider truncation on the $\pi^*$-intersection graph $\mathcal H_{\pi^*}$ sampled according to $\mathsf Q$ (i.e., when the two graphs are correlated). 

We learned  from \cite{WXY21+} that when computing the moments of the likelihood ratio, the concept of \emph{edge orbits} (induced by a permutation) plays an important role. We streamline this intuition a little further and prove Lemma~\ref{lem:le1} in Section~\ref{subsec:conditional-second-moment-method}. In light of Lemma~\ref{lem:le1}, it is natural to separate the edge orbits depending on whether they are entirely contained in the $\pi^*$-intersection graph $\mathcal H_{\pi^*}$ (see \eqref{eq-def-J-sigma}). With  Lemma~\ref{lem:le1} at hand, the issue reduces to bounding the moment from edge orbits that are entirely in $\mathcal H_{\pi^*}$, which naturally calls for a truncation on $\mathcal H_{\pi^*}$. As a major difference between \cite{WXY21+} and our work, instead of truncating  $\mathcal H_{\pi^*}$ as a pseudo forest as in \cite{WXY21+}, we truncate on the maximal subgraph density for $\mathcal H_{\pi^*}$ and some other mild conditions on small subgraph counts in $\mathcal H_{\pi^*}$ (see Section~\ref{subsec:on-good-event}). In Section~\ref{subsec:exp-moment} we bound the truncated moment, where the truncation on the maximal subgraph density plays a crucial role since it rules out the possibility of creating a large number of edge orbits in $\mathcal H_{\pi^*}$ by only fixing the values of the bijection on a small number of vertices.

\subsection{The conditional second moment method}\label{subsec:conditional-second-moment-method}

Let $\mathcal{Q}$ be the probability measure on the sample space $\Omega=\{(G,\mathsf G,\pi^*)\}$ under $\operatorname{H}_1$, i.e. 
under $\mathcal{Q}$, $\pi^*$ is a uniform bijection and conditioned on $\pi^*$, $(G,\mathsf G)$ is a pair of correlated graphs sampled according to rule \eqref{eq:rule}. Note that $\mathsf Q$ is nothing but the marginal distribution on the first two coordinates of $\mathcal{Q}$. Our goal is to show $\operatorname{TV}(\mathsf P,\mathsf Q)=o(1)$. To this end, we consider the likelihood ratio

\begin{equation*}\label{eq-likelihood-ratio}
	L(G,\mathsf G)\stackrel{\operatorname{\Delta}}{=}\frac{\dif\mathsf Q}{\dif\mathsf P}\mid_{(G,\mathsf G)}=\frac{\sum_{\pi}\mathcal{Q}[G,\mathsf G,\pi]}{\mathsf{P}[G,\mathsf G]}=\frac{1}{n!}\sum_{\pi}\frac{\mathcal{Q}[G,\mathsf G\mid \pi]}{\mathsf P[G,\mathsf G]}\,,
\end{equation*}
where $\mathcal{Q}[G,\mathsf G\mid \pi]$ is a short notation for $\mathcal{Q}[G,\mathsf G\mid \pi^* =  \pi]$.
Ideally we wish to show $\mathbb{E}_\mathsf{Q}L=\mathbb{E}_\mathsf{P}L^2=1+o(1)$, but this fails due to the contribution from certain rare event. Therefore, we turn to the conditional moment for the likelihood ratio where we choose some ``good'' event $\mathcal{G}$  measurable with respect to $(G,\mathsf G,\pi^*)$ such that $\mathcal{Q}[\mathcal G]=1-o(1)$. 
Let $\mathcal{Q}'[\cdot]=\mathcal{Q}[\cdot\mid \mathcal{G}]$ and $\mathsf Q'$ be the marginal distribution of the first two coordinates of $\mathcal{Q}'$. We further assume that the marginal distribution of $\pi^*$ under $\mathcal{Q}'$ is still uniform, then the conditional likelihood ratio is then given by 
\begin{equation*}
	L'(G,\mathsf G)\stackrel{\operatorname{\Delta}}{=}\frac{\dif\mathsf Q'}{\dif\mathsf P}\mid_{(G,\mathsf G)}=\frac{\sum_{\pi}\mathcal{Q'}[G,\mathsf G,\pi]}{\mathsf{P}[G,\mathsf G]}=\frac{1}{n!}\sum_{\pi}\frac{\mathcal{Q'}[G,\mathsf G\mid \pi]}{\mathsf P[G,\mathsf G]}\,.
\end{equation*}
By the data processing inequality and the assumption that $\mathcal Q[\mathcal G]=1-o(1)$,
\begin{equation*}
	\operatorname{TV}(\mathsf Q,\mathsf Q')\le \operatorname{TV}(\mathcal Q,\mathcal Q')=o(1)\,.
\end{equation*}
As a result, once we show that $\mathbb{E}_{\mathsf Q'}L'=\mathbb{E}_\mathsf{P}L'^2=1+o(1)$ for some appropriately chosen good event $\mathcal{G}$, then by the triangle inequality,
\begin{equation*}\label{eq:bound-1}
	\begin{aligned}
		\operatorname{TV}(\mathsf P,\mathsf Q)\le &\operatorname{TV}(\mathsf P,\mathsf Q')+\operatorname{TV}(\mathsf Q',\mathsf Q)
		\le \sqrt{\log \mathbb{E}_{\mathsf Q'}L'}+o(1)=o(1)\,,
	\end{aligned}
\end{equation*}
where the second inequality follows from  Pinsker's inequality and Jensen's inequality.

As we will see later, our good event $\mathcal G$ will be measurable with respect to the isomorphic class of the $\pi^*$-intersection graph of $G$ and $\mathsf G$, thus $\pi^*$ does have uniform distribution under $\mathcal{Q}'$. In addition, in what follows, all the probability analysis conditioned on $\pi^*$ is invariant with the realization of $\pi^*$. In particular, we can write 
$$\E_\mathsf {Q'} L' = \E_{\pi^*\sim \mathcal {Q'}} \E_{(G, \mathsf G)\sim \mathcal Q'[\cdot \mid \pi^*]} L',$$ 
where $\E_{(G, \mathsf G)\sim \mathcal Q'[\cdot \mid \pi^*]} L'$ is invariant of the realization of $\pi^*$. As a result, for convenience of exposition, in what follows we regard $\pi^*$ as certain fixed bijection from $V$ to $\mathsf V$ in order to avoid unnecessary complication from another layer of randomness.

We postpone the definition of our good event $\mathcal{G}$ in Section~\ref{subsec:on-good-event}, and we next investigate the conditional likelihood ration $L'(G,\mathsf G)$ more carefully. 
For any bijection $\pi:V\to\mathsf V$, it is easy to see
\begin{equation}\label{eq:likelihood-ratiooo}
	\frac{\mathcal Q[G,\mathsf G\mid \pi]}{\mathsf P[G,\mathsf G]}=\prod_{(u,v)\in E_0}\ell(G_{u,v},\mathsf G_{\pi(u),\pi(v)})\,,
\end{equation}
where $\ell:\{0,1\}\times \{0,1\}\to \mathbb{R}$ denotes for the likelihood ratio function for a pair of edges given by
\begin{equation}\label{eq:function-ell}
	\ell(x,y)=\begin{cases}
		\frac{1-2ps+ps^2}{(1-ps)^2}\quad &x=y=0\,;\\
		\frac{1-s}{1-ps}\quad &x=1, y=0\text{ or }x=0, y=1\,;\\
		\frac{1}{p}&x=y=1\,.
	\end{cases}
\end{equation}

Now suppose we are under $\operatorname{H}_1$. For a bijection $\pi:V\to \mathsf V$, define a permutation on $V$ by $\sigma\stackrel{\Delta}{=}\pi^{-1}\circ \pi^*$. Then $\pi$ (respectively $\sigma$) induces a bijecion $\Pi$ from $E_0$ to $\mathsf E_0$ (respectively a permutation $\Sigma$ on $E_0$), given by $\Pi((u,v))=(\pi(u),\pi(v))$ (respectively $\Sigma((u,v))=(\sigma(u),\sigma(v))$).
For a given permutation $\sigma$ on $V$, let $\mathcal{O}_\sigma$ be the set of edge orbits induced by $\Sigma$, and we define
\begin{equation}\label{eq-def-J-sigma}
	\begin{aligned}
		\mathcal{J}_\sigma=
		\{O\in \mathcal{O}_\sigma:G_{u,v}=\mathsf G_{\pi^*(u),\pi^*(v)}=1,\mbox{ for all }(u,v)\in O\}
	\end{aligned}
\end{equation}
to be the set of edge orbits that are \emph{entirely} contained in $\mathcal{H}_{\pi^*}$  (recall \eqref{def:pi-intersection-graph} for the definition of $\pi^*$-intersection graph $\mathcal{H}_{\pi^*}$). 
Note that $\mathcal{O}_\sigma$ is deterministic for a given $\sigma$, while $\mathcal{J}_\sigma$ is random depending on $\pi^*$ (which was assumed to be fixed) and the realization of $(G,\mathsf G)$. Let $H(\mathcal J_\sigma)$ be the subgraph of $\mathcal{H}_{\pi^*}$ with vertices and edges from orbits in $\mathcal{J}_\sigma$ and with slight abuse of notation, we denote by $|\mathcal{J}_\sigma|$ the number of edges in $H(\mathcal J_\sigma)$.

It is clear that once $\pi^*$ and $\sigma=\pi^{-1}\circ \pi^*$ are fixed, the collections of $\{\ell(G_{e},\mathsf{G}_{\Pi(e)}): e \in O\}$ are mutually independent for $O \in \mathcal{O}_\sigma$. In addition, for $O \in \mathcal J_\sigma$ we have
\begin{equation}\label{eq:entire-orbit}
	\prod_{e \in O}\ell(G_{e},\mathsf{G}_{\Pi(e)})=p^{-|O|}\,,
\end{equation}
where $|O|$ denotes for the number of edges $e\in O$. The contribution to the untruncated moment from $O\in \mathcal J_\sigma$ blows up and thus calls for a truncation. Before doing that, we first prove the following lemma (similar to \cite[Proposition 3]{WXY20+}) which controls contribution from $O\notin \mathcal{J}_\sigma$.
\begin{lemma} \label{lem:le1}
	Let $\pi=\pi^*\circ \sigma^{-1}$ (the notations for $\pi^*, \pi, \sigma$ and $\Pi, \Sigma$ are consistent as above). For $p,s\le 0.1$ and $O\in \mathcal{O}_\sigma$,
	\begin{equation}\label{eq:lem1}
		\mathbb{E}_{(G,\mathsf G)\sim \mathcal{Q}[\cdot\mid \pi^*]}\Big[\prod_{e\in O}\ell(G_{e},\mathsf{G}_{\Pi(e)})\mid O\notin  \mathcal{J_\sigma}\Big]\le 1\,.
	\end{equation}
\end{lemma}
\begin{proof}
	We first compute the expectation without conditioning on $O\not\in \mathcal J_\sigma$ as follows:
	\begin{equation*}
		\mathbb{E}_{(G,\mathsf G)\sim \mathcal Q[\cdot\mid \pi^*]}\prod_{e\in O}\ell(G_e,\mathsf G_{\Pi(e)})=\mathbb{E}_{(G,\mathsf G)\sim \mathsf{P}}\prod_{e\in O}\ell(G_e,\mathsf G_{\Pi(e)})\ell(G_{\Sigma(e)},\mathsf G_{\Pi(e)})\,.
	\end{equation*}
	The right hand side above can be further interpreted as $\operatorname{Tr}(\mathcal{L}^{2|O|})$, where $\mathcal{L}$ is the integral operator on the space of real functions on $\{0,1\}$ induced by the kernel $\ell(\cdot,\cdot)$ in \eqref{eq:function-ell} as
	\begin{equation*}
		(\mathcal{L}f)(x)\stackrel{\Delta}{=}\mathbb{E}_{\mathsf G_{\mathsf e}\sim \mathsf P}[\ell(x,\mathsf G_{\mathsf e})f(\mathsf G_{\mathsf e})]=\mathbb{E}_{(G_e,\mathsf G_{\mathsf e})\sim \mathsf Q}[f(\mathsf G_\mathsf{e})\mid G_e=x]\,.
	\end{equation*}
	The matrix form of $\mathcal{L}$ is given by $M(x,y)=\ell(x,y)\cdot\mathsf P[\mathsf G_{\mathsf e}=y]$ for$(x,y)\in \{0,1\}\times\{0,1\}$, which can be written explicitly as
	\begin{equation*}
		M=\left(\begin{matrix}
			\frac{1-2ps+ps^2}{1-ps}\ &\frac{ps(1-s)}{1-ps}\\
			1-s\ &s
		\end{matrix}\right)\,.
	\end{equation*} 
	$M$ has two eigenvalues $1$ and $\rho\stackrel{\Delta}{=}\frac{s(1-p)}{1-ps}$, so the unconditional expectation equals to $1+\rho^{2|O|}$. See also \cite[Proposition 1]{WXY20+} for details.
	
	Since conditioned on $O \notin \mathcal{J}_\sigma$ only excludes the case that $G_e=\mathsf G_{\Pi(e)}=1$ for all  $e\in O$, 
	\begin{equation*}
		\mathbb{E}_{(G,\mathsf G)\sim \mathcal{Q}[\cdot\mid \pi^*]}\Big[\prod_{e\in O}\ell(G_{e},\mathsf{G}_{\Pi(e)})\mid O\notin  \mathcal{J_\sigma}\Big]=\frac{1+\rho^{2|O|}-s^{2|O|}}{1-(ps^2)^{|O|}}\le 1\,,
	\end{equation*} 
	where the last inequality follows because for any $0<p,s\le 0.1$ and $k\ge 1$,
	\begin{equation*}
		s^{2k}-\rho^{2k}=s^{2k}\Big(1-\left(\frac{1-p}{1-ps}\right)^{2k}\Big)\ge s^{2k}\Big(1-(1-p)^k\Big)\ge (ps^2)^{k}. \qedhere
	\end{equation*}
\end{proof}
We are now ready to derive the next lemma.
\begin{lemma}\label{lem-exp-moment}
	With notations in this subsection, we have
	\begin{equation}\label{eq:exp-moment}
		\mathbb{E}_{(G,\mathsf G)\sim \mathsf Q'}L'(G,\mathsf G) \leq \frac{1}{\mathcal Q[\mathcal G]}\mathbb{E}_{(G,\mathsf G,\pi^*)\sim \mathcal{Q}'}\frac{1}{n!}\sum_{\sigma}\frac{p^{-|\mathcal{J}_\sigma|}}{\mathcal Q[\mathcal G\mid\pi^*,  \mathcal J_\sigma]}\,.
	\end{equation}
\end{lemma}
\begin{proof}
	First note that $\mathcal Q'[G,\mathsf G,\pi]\le \frac{\mathcal Q[G,\mathsf G,\pi]}{\mathcal Q[\mathcal G]}$ holds for any triple $(G,\mathsf G,\pi)$. Thus,
	\begin{align}
		&\mathbb{E}_{(G,,\mathsf G)\sim \mathsf Q'}L'(G,\mathsf G)=\mathbb{E}_{(G,\mathsf G)\sim \mathsf Q'}\sum_{\pi}\frac{\mathcal Q'[G,\mathsf G,\pi]}{\mathsf P[G,\mathsf G]}\nonumber\\
		\le&\ \frac{1}{\mathcal Q[\mathcal G]}\mathbb{E}_{(G,\mathsf G)\sim \mathsf Q'}\sum_{\pi}\frac{\mathcal Q[G,\mathsf G,\pi]}{\mathsf P[G,\mathsf G]}
		\stackrel{\eqref{eq:likelihood-ratiooo}}{=} \frac{1}{\mathcal Q[\mathcal G]}\mathbb{E}_{(G,\mathsf G)\sim\mathsf Q'}\frac{1}{n!}\sum_{\pi}\prod_{e\in E_0}\ell(G_{e},\mathsf G_{\Pi(e)})\nonumber\\
		=&\ \frac{1}{\mathcal Q[\mathcal G]} \mathbb{E}_{\pi^*\sim \mathcal{Q}'}\frac{1}{n!}\sum_{\sigma}\mathbb{E}_{(G,\mathsf G)\sim \mathcal{Q}'[\cdot\mid \pi^*]}\prod_{O\in \mathcal{O}_\sigma}\prod_{e\in O}\ell(G_e,\mathsf G_{\Pi(e)})\,,\label{eq:sum-over-all-pi}
	\end{align}
	where in the last equity we changed from summation over $\pi$ to summation over $\sigma=\pi^{-1}\circ\pi^{*}$. By Lemma~\ref{lem:le1}, we can take conditional expectation with respect to $\mathcal J_\sigma$ for each $\sigma$ and obtain that
	\begin{align}
		&\ \mathbb{E}_{(G,\mathsf G)\sim \mathcal{Q}'[\cdot\mid \pi^*]}\prod_{O\in \mathcal{O}_\sigma}\prod_{e\in O}\ell(G_e,\mathsf G_{\Pi(e)})\nonumber\\
		\stackrel{\eqref{eq:entire-orbit}}{=}&\ \mathbb{E}_{\mathcal{J}_\sigma\sim \mathcal{Q'}[\cdot\mid \pi^*]}\left[p^{-|\mathcal{J}_\sigma|}\mathbb{E}_{(G,\mathsf G)\sim \mathcal{Q'}[\cdot\mid \pi^*,\mathcal{J_\sigma}]}\left[\prod_{O\in \mathcal{O}_\sigma\setminus \mathcal{J}_\sigma}\prod_{e\in O}\ell(G_e,\mathsf G_{\Pi(e)})\mid \mathcal{J}_\sigma\right]\right]\,. \label{eq-condition-on-J-sigma}
	\end{align}
	Note that for each fixed $\sigma$, and any realization $J$ of $\mathcal{J}_\sigma$, 
	\begin{align}\label{eq-Q-to-Q'}
		&\ \mathbb{E}_{(G,\mathsf G)\sim\mathcal{Q}'[\cdot\mid \pi^*,\mathcal{J_\sigma}=J]}\left[\prod_{O\in \mathcal{O}_\sigma\setminus \mathcal {J}_\sigma}\prod_{e\in O}\ell(G_e,\mathsf G_{\Pi(e)})\mid \mathcal J_\sigma=J\right]\nonumber\\
		\le&\ \frac{1}{\mathcal Q[\mathcal{G}\mid \pi^*,\mathcal J_\sigma=J]}\mathbb{E}_{(G,\mathsf G)\sim\mathcal{Q}[\cdot\mid \pi^*]}\left[\prod_{O\in \mathcal{O}_\sigma\setminus J}\prod_{e\in O}\ell(G_e,\mathsf G_{\Pi(e)})\mid \mathcal{J}_\sigma=J\right]\,,
	\end{align}
	where the term $\frac{1}{\mathcal Q[\mathcal{G}\mid \pi^*,\mathcal J_\sigma=J]}$ emerges since we moved from $\mathcal Q'$ to $\mathcal Q$. (Note that although $\mathcal Q$ and $\mathcal Q'$ are similar since presumably $\mathcal Q[\mathcal G] = 1-o(1)$, the term $\frac{1}{\mathcal Q[\mathcal{G}\mid \pi^*,\mathcal J_\sigma=J]}$ is not necessarily negligible since conditioned on $\mathcal J_\sigma=J$ may substantially decrease the probability for $\mathcal G$.) By Lemma~\ref{lem:le1}, we have that
	\begin{align*}
		\quad &\mathbb{E}_{(G,\mathsf G)\sim\mathcal{Q}[\cdot\mid \pi^*]}\left[\prod_{O\in \mathcal{O}_\sigma\setminus J}\prod_{e\in O}\ell(G_e,\mathsf G_{\Pi(e)})\mid \mathcal{J}_\sigma=J\right]\\
		=& \prod_{O\in \mathcal{O}_\sigma\setminus J}\mathbb{E}_{(G,\mathsf G)\sim\mathcal Q[\cdot\mid \pi^*]}\left[\prod_{e\in O}\ell(G_e,\mathsf G_{\Pi(e)})\mid O\notin \mathcal{J}_\sigma\right] \leq 1\,.
	\end{align*}
	Combined with \eqref{eq:sum-over-all-pi},  \eqref{eq-condition-on-J-sigma} and \eqref{eq-Q-to-Q'}, it yields that $\mathbb{E}_{(G,\mathsf G)\sim \mathsf Q'}L'(G,\mathsf G)$ is upper-bounded by
	\begin{equation*}
		\frac{1}{\mathcal G[\mathcal G]}\mathbb{E}_{\pi^*\sim \mathcal{Q'}}\frac{1}{n!}\sum_{\sigma}\mathbb{E}_{\mathcal{J}_\sigma\sim \mathcal{Q'}[\cdot\mid \pi^*]}\frac{p^{-|\mathcal{J}_\sigma|}}{\mathcal Q[\mathcal G\mid \pi^*,\mathcal J_\sigma]}=\frac{1}{\mathcal Q[\mathcal G]}\mathbb{E}_{(G,\mathsf G,\pi^*)\sim \mathcal{Q}'}\frac{1}{n!}\sum_{\sigma}\frac{p^{-|\mathcal{J}_\sigma|}}{\mathcal Q[\mathcal G\mid\pi^*,  \mathcal J_\sigma]}\,,
	\end{equation*}
	as required.
\end{proof}

\subsection{Truncation for the $\pi^*$-intersection graph}\label{subsec:on-good-event}
In light of Lemma~\ref{lem-exp-moment}, it suffice to show that the right hand side of \eqref{eq:exp-moment} is upper-bounded by $1+o(1)$ for some appropriately chosen event $\mathcal G$ with $\mathcal Q[\mathcal G]=1-o(1)$. 
Under $\operatorname{H}_1$, the $\pi^*$-intersection graph $\mathcal{H}_{\pi^*} = (V,\mathcal E_{\pi^*})$ of $G$ and $\mathsf G$ is an Erd\H{o}s-R\'enyi graph $\mathbf{G}(n,\frac{\lambda}{n})$, where $\lambda=nps^2$. Recall that we are now under the assumption
\begin{equation}\label{eq-impossibility-assumption}
	p=n^{-\alpha+o(1)}\mbox{ for  some }\alpha<1 \mbox{ and } \lambda\le \lambda_*-\varepsilon \mbox{ for some constant }\varepsilon>0\,.
\end{equation}

In this subsection, we will define our good event $\mathcal G$. To this end, we need some more notations. For simple graphs $H$ and $\mathcal H$, a \emph{labeled embedding} of $H$ into $\mathcal H$ is an injective map $\tau:H\to \mathcal H$, such that  $(\tau(u),\tau(v))$ is an edge of $\mathcal H$ when $(u,v)$ is an edge of $H$. Further, we define an \emph{unlabeled embedding} of $H$ into $\mathcal H$ to be an equivalent class of labeled embeddings: for two labeled embeddings $\tau_1,\tau_2:H\to\mathcal H$, we say $\tau_1\sim\tau_2$ if and only if there exists an automorphism $\phi:H\to H$, such that $\tau_2=\tau_1\circ \phi$. 

Let $t(H,\mathcal H)$ be the number of unlabeled embeddings of $H$ into $\mathcal H$. Then it is clear that the number of labeled embeddings of $H$ into $\mathcal H$ is given by $\operatorname{Aut}(H)t(H,\mathcal H)$, where $\operatorname{Aut}(H)$ stands for the number of automorphsims of $H$ to itself. For each isomorphic class $\mathcal C$, pick a representative element $H_{\mathcal C}\in \mathcal C$ and fix it. For each $k\ge 2$, let $\mathfrak C_{k} =  \mathfrak C_{k}(\mathcal H)$ (respectively $\mathfrak T_k = \mathfrak T_k(\mathcal H)$) be the collection of all such representatives $H_{\mathcal C}$ which are connected non-tree graphs (respectively trees) with $k$ vertices so that $t(H_{\mathcal C}, \mathcal H) \geq 1$.

	Denote $\xi=\frac{1}{2}[\varrho(\lambda_*-\varepsilon)+\varrho(\lambda_*)]$. Since $\lambda_*>1$ in the case $\alpha<1$, by Proposition~\ref{prop:densest-subgraph-is-linear} we have $\varrho(\lambda_*-\varepsilon)<\xi<\varrho(\lambda_*)=\frac{1}{\alpha}$. Hence, when $n$ is large enough, 
	\begin{equation}\label{eq-xi-delta-assumption}
		np^\xi\ge n^{\delta_0} \mbox{ for some constant }\delta_0>0\,.
	\end{equation} 
	Fix some positive constant $\delta<\delta_0$, we say that a graph $\mathcal H = (V, \mathcal E)$ is \emph{admissible} if it satisfies the following properties:
	\begin{enumerate}
		\item[(i)] The maximal edge-vertex ratio over all subgraphs does not exceed $\xi$, i.e.
		\begin{equation}\label{eq:(i)-in-admissibility}
			\max_{\emptyset \neq U\subset V} \frac{|\mathcal E(U)|}{|U|} \leq \xi\,.
		\end{equation}
		\item[(ii)] The maximal degree of $\mathcal{H}$ is no more than $\log n$. 
		\item[(iii)] Any connected subgraph containing at least two cycles has size larger than $2\log \log n$.
		\item[(iv)] For any $k\ge 2$, the number of $k$-cycles is bounded by $n^{\delta k}$.
	\end{enumerate}
	Define the good event $\mathcal{G}=\{\mathcal{H}_{\pi^*}\text{ is }\text{admissible}\}$. 
	\begin{lemma}\label{lem:good-event-probability}
		For an Erd\H{o}s-R\'enyi graph $\mathcal{H}\sim \mathbf{G}\left(n,\frac{\lambda}{n}\right)$,
		\begin{equation}\label{eq:probability-of-admissibility}
			\P[\mathcal H\text{ is admissible}]=1-o(1)\,.
		\end{equation} 
		In addition, there exists a constant $c=c(\lambda,\delta)>0$, such that for any subgraph $H$ satisfies $\P[\mathcal H \text{ is admissible}\mid  H\subset \mathcal H]>0$ and any event $\mathcal F$ that is measurable with respect to and decreasing with edges that are not contained in $H$, we have
		\begin{equation}\label{eq:probability-of-admissible-condition-on-subgraph-1}
			\P[\mathcal H\text{ is admissible}\mid  H\subset \mathcal H, \mathcal F]\ge [1-o(1)]c^{|E(H)|}\,.
		\end{equation} 
	\end{lemma}
	\begin{remark}
		By \eqref{eq:probability-of-admissibility}, we see that $\mathcal Q[\mathcal G]=1-o(1)$. In addition, conditioning on $\mathcal J_\sigma=J$ for some realization $J$ sampled from $\mathcal Q[\cdot\mid \pi^*,\mathcal G]$, we have $\mathcal{Q}[\mathcal G\mid \pi^*,\mathcal J_\sigma=J]>0$. Since $\{\mathcal J_\sigma=J\} = \{H(J)\subset \mathcal H_{\pi^*}\} \cap \mathcal F$ where $\mathcal F$ is the event that there is no other edge orbit (except those in $J$) that is entirely contained in $\mathcal H_{\pi^*}$, we see from \eqref{eq:probability-of-admissible-condition-on-subgraph-1} that 
		\begin{equation}\label{eq:probability-of-admissible-condition-on-subgraph-2}
			\mathcal{Q}[\mathcal G\mid \pi^*,\mathcal J_\sigma=J]\ge[1-o(1)]c^{|E\left(H(J)\right)|}=[1-o(1)]c^{|\mathcal{J}_\sigma|}\,.
		\end{equation}
		The preceding inequality is useful for us since on the right hand side of \eqref{eq:exp-moment} there is a term of $\mathcal Q[\mathcal G\mid \pi^*,\mathcal J_\sigma]$.
	\end{remark}
	\begin{proof}[Proof of Lemma~\ref{lem:good-event-probability}]
		First we show \eqref{eq:probability-of-admissibility}. It suffices to bound the probability that either of (i)-(iv) fails. For (i), since $\varrho(\lambda)\le\varrho( \lambda_*-\varepsilon)<\xi$, Proposition~\ref{prop:densest-subgraph} gives  $\P[(i)\text{ fails}]=o(1)$. $\P[(ii) \text{ or }(iii)\text{ fails}]=o(1)$ is well-known. Indeed, the typical value of the maximal degree in $\mathcal H$ is of order $\frac{\log n}{\log\log n}$ (see e.g. \cite[Theorem 3.4]{AM22}), and the typical value for the minimal size of connected subgraphs containing at least two cycles in $\mathcal H$ is at least of order $\log n$ (this can be shown by a simple union bound). For (iv), since the expected number of $k$-cycles in $\mathcal H$ is bounded by $\lambda^k$,  by Markov inequality we get that $\P[(iv) \text{ fails}]\le \sum_{k=1}^{\infty}\frac{\lambda^k}{n^{\delta k}}=o(1)$. Altogether, this yields \eqref{eq:probability-of-admissibility}.
		
		For \eqref{eq:probability-of-admissible-condition-on-subgraph-1}, the case $H=\emptyset$ reduces to \eqref{eq:probability-of-admissibility} by FKG inequality (since $\{\mathcal H\text{ is admissible}\}$ is a decreasing event), so we can assume $ H\neq \emptyset$. The condition $\P[\mathcal H \text{ is admissible }\mid  H\subset \mathcal H]>0$ implies that the subgraph $H$ satisfies all conditions in admissibility. Let $V_1$ be the vertex set of $ H$, and $V_2=V\setminus V_1$. Consider the following three events:
		\begin{equation*}
			\begin{aligned}
				&\mathcal A_1=\{\mbox{There is no edge within $V_1$ except those in $ H$},\, \mathcal{E}(V_1,V_2)=\emptyset\},\\
				&\mathcal{A}_2=\{\text{There is no cycle in }V_2\mbox{ with length less than }K\stackrel{\Delta}{=}\lceil{\delta^{-1}}\rceil\}\,,\\
				&\mathcal{A}_3=\{\text{The subgraph in }V_2\mbox{ satisfies (i), (ii), (iii) in admissibility}\}\,.\\
			\end{aligned}
		\end{equation*}
		We claim that conditioned on $ H\subset \mathcal H$, we have $\mathcal H$ is admissible as long as  $\mathcal A_1\cap \mathcal A_2\cap \mathcal A_3$ holds (we comment that the purpose of defining event $\mathcal A_2$ is to handle the potential scenario where there are $n^{\delta k}$ $k$-cycles in $H$). Clearly, $\mathcal H$ satisfies (i), (ii), (iii) in admissibility by $\mathcal A_1\cap \mathcal A_3$. For (iv), the case $k<K$ is guaranteed by $\mathcal A_2$. When $K\le k\le \log\log n$, no two $k$-cycle share a common vertex by (iii), thus the number of $k$-cycles is at most $n/k<n^{\delta k}$ since $k\ge K$; and when $k\ge \log \log n$, the number of $k$-cycles is bounded by $n(\log n)^k\le n^{\delta k}$ from (ii). Thus (iv) also holds and $\mathcal H$ is admissible.
		
		Note that $\mathcal{A}_1$ is independent with $\mathcal A_2\cap \mathcal A_3$, $\P[\mathcal A_1]\ge e^{-2\lambda| E(H)|}$, and  $\P[\mathcal A_2]>C_\delta$ for some constant $C_\delta>0$ since the distribution of small cycles are approximately independent Poisson variables (one can also use FKG inequality instead of approximate independence here since the number of small cycles are all increasing with the graph). Since the subgraph within $V_2$ is an Erd\H{o}s-R\'enyi, we get $\P[\mathcal A_3^c]=o(1)$ from \eqref{eq:probability-of-admissibility}. Therefore, for some small constant $c=c(\lambda,\delta)>0$ we have
		\begin{align*}
			&\P[\mathcal H\text{ is admissible}\mid  H\subset \mathcal H, \mathcal F]\ge \P[\mathcal A_1\cap \mathcal A_2\cap \mathcal A_3 \mid \mathcal F]\ge \P[\mathcal A_1\cap \mathcal A_2\cap \mathcal A_3] \\
			&\geq \P[\mathcal A_1]\left(\P[\mathcal A_2]-\P[\mathcal A_3^c]\right)
			\ge[C_\delta-o(1)] e^{-2\lambda| E(H)|}\ge [1-o(1)]c^{| E(H)|}\,,
		\end{align*}
		where we applied FKG inequality for the last transition in the first line (note that $\mathcal A_1, \mathcal A_2, \mathcal A_3$ are all decreasing events). This completes the proof of the lemma.
	\end{proof}
	
	As a result of admissibility, we have the following bounds on subgraph counts.
	\begin{lemma}\label{lem:subgraph-counts}
		For an admissible graph $\mathcal H$ and for $k\geq 2$, we have that the total number of labeled embeddings of $T\in \mathfrak{T}_k$ is bounded by $n(4\log n)^{2(k-1)}$, i.e.
		\begin{equation}\label{eq:count-of-tree}
			\sum_{T\in \mathfrak T_k}\operatorname{Aut}(T)t(T,\mathcal H)\le n(4\log n)^{2(k-1)}\,.
		\end{equation}
		In addition, the total number of labeled embeddings of $C\in \mathfrak{C}_k$ is bounded by $k^3(2^{\xi+1}n^\delta)^{k}$, i.e. 
		\begin{equation}\label{eq:count-of-non-tree}
			\sum_{C\in \mathfrak C_k}\operatorname{Aut}(C)t(C,\mathcal H)\le k^3(2^{\xi+1}n^\delta)^{k}\,.
		\end{equation}
	\end{lemma}
	\begin{proof}
		By \cite{Otter48}, the number for isomorphic classes of trees with $k$ vertices is at most $4^{k-1}$. For each class, we claim that the number of labeled embedding is at most $n(\log n)^{2(k-1)}$. This is because each embedding can be encoded by a walk path of length $2(k-1)$ on the graph which corresponds to the depth-first search contour of the image of the embedding. By (ii) in admissibility, the number of paths of length $2(k-1)$ is at most $n(\log n)^{2(k-1)}$. This yields  \eqref{eq:count-of-tree}.
		
		For \eqref{eq:count-of-non-tree}, note that every labeled non-tree subgraph with $k$ vertices on $\mathcal H$ can be constructed by the following steps: \\
		{\bf Step 1.} Pick an isomorphic class of connected graphs with $k$ vertices and $k$ edges, and take its representative $C$ with vertices labeled by $v_1,\dots,v_k$.\\
		{\bf Step 2.} Choose a labeled embedding $\tau:C\to \mathcal H$.\\
		{\bf Step 3.} Add some of the remaining edges within $\{\tau(v_1),\dots,\tau(v_k)\}$ to $\tau(C)$ and get the final subgraph.  
		
		The number of isomorphic classes in {\bf Step 1} is no more than $k^24^{k-1}$ since we can first pick a tree of $k$ vertices and then add an extra edge. For any connected subgraph $C$ with $k$ vertices and $k$ edges, $C$ is a union of a cycle with $r\le k$ vertices together with some trees. The number of labeled embeddings of the cycle is bounded by $2r(n^\delta)^r$ by (iv) in admissibility, and once this is done, the number of ways to embed the rest of trees is bounded by $(\log n)^{2(k-r)}\le n^{\delta(k-r)}$ from (ii) in admissibility (and a similar argument as for \eqref{eq:count-of-tree}). So, the number of labeled embedding in {\bf Step 2} is bounded by $2k(n^{\delta})^k$. Finally, for any labeled embedding $\tau:C\to\mathcal H$, since (i) holds, the total number of remaining edges between $\tau(v_1),\dots,\tau(v_k)$ is bounded by $(\xi-1)k$, we see that the number of choices for {\bf Step 3} is at most $2^{(\xi-1)k}$. Now a simple application of multiplication rule yields \eqref{eq:count-of-non-tree}.
	\end{proof}
	
	\subsection{The truncated exponential moment}\label{subsec:exp-moment}
	We now prove the following bound.
	\begin{proposition}\label{prop:control-exp-moment}
		Suppose \eqref{eq-xi-delta-assumption} holds and suppose that $\mathcal H$ is admissible. Then as $n\to\infty$,
		\begin{equation}\label{eq-deterministic-exp-moment}
			\frac{1}{n!} \sum_\sigma\frac{p^{-|\mathcal J_\sigma|}}{\mathcal Q[\mathcal G\mid \pi^*,\mathcal J_\sigma]} \le1+o(1)\,.
		\end{equation}
	\end{proposition}
	Combined with Lemmas~\ref{lem-exp-moment} and \ref{lem:good-event-probability} as well as discussions at the beginning of Section~\ref{subsec:conditional-second-moment-method}, this then completes the proof of \eqref{eq:negative-<1}.
	
	When proving \eqref{eq-deterministic-exp-moment}, it would be convenient to first fix a subgraph $H$ and sum over all permutations $\sigma$ with $H(\mathcal{J}_\sigma) = H$, and then sum over all possible $H$. To this end, we will need the following lemma on an upper bound for the number of permutations $\sigma$ on $V$ such that $H(\mathcal J_\sigma)\cong H_{\mathcal C}$ for each possible isomorphic class representative $H_\mathcal{C}$ that can arise from an admissible graph and some permutation.  Note that for each realization of $H(\mathcal J_\sigma)$, the collection of its components is isomorphic to a union of some $C$'s in $\bigcup \mathfrak{C}_r$ and some $T$'s in $\bigcup \mathfrak{T}_s$, since $H(\mathcal{J}_\sigma)$ is always a subgraph of $\mathcal H$.
	\begin{lemma}
		Suppose the collection of components of $H_\mathcal{C}$ is isomorphic to a union of some $C_{i}\in \mathfrak C_{r_i}$ with $x_i$ copies for $1\leq i\leq l$ and some $T_{j}\in \mathfrak T_{s_j}$ with $y_j$ copies for $1\leq j\leq m$, where $C_1,\dots,C_l,T_1,\dots,T_m$ are in distinct isomorphic classes. Then
		\begin{align}\label{eq:number-of-sigma}
			|\sigma:H(\mathcal{J}_\sigma)\cong H_\mathcal{C}|\le&\big(n-\sum_{i\in [l]}x_ir_i-\sum_{j\in [m]}y_js_j\big)! \nonumber\\ \times&\prod_{i\in [l]}\left(\operatorname{Aut}(C_i) t(C_i,\mathcal{H})\right)^{x_i}\prod_{j\in [m]}\left(\operatorname{Aut}(T_j) t(T_j,\mathcal H)\right)^{y_j}\,.
		\end{align}
	\end{lemma}
	\begin{proof}
		First we choose an unlabeled embedding $\tau$ for $H_{\mathcal C}$ in $\mathcal H$. Since  $C_1,\dots,C_l,T_1,\dots,T_m$ are in distinct isomorphic classes, $\tau$ can be viewed as a product of $\tau|_{x_i\cdot C_i}$ and $\tau|_{y_j\cdot T_j}$ for $1\leq i\leq l$ and $1\leq j\leq m$, where each $\tau|_{x_i\cdot C_i}$ (respectively $\tau|_{y_j\cdot T_j}$) is an unlabeled embedding for $x_i$ disjoint copies of $C_i$ (respectively $y_j$ disjoint copies of $T_j$). Therefore, the number of choices for $\tau$ is bounded by 
		\begin{eqnarray}\label{eq:location-choice}
			\prod_{i\in [l]}\binom{t(C_i,\mathcal H)}{x_i}\prod_{j\in [m]}\binom{t(T_j,\mathcal H)}{y_j}\le \prod_{i\in [l]}\frac{t(C_i,\mathcal H)^{x_i}}{x_i!}\prod_{j\in [m]}\frac{t(T_j,\mathcal H)^{y_j}}{y_j!}\,.
		\end{eqnarray}
		For each unlabeled embedding $\tau$, we wish to bound the number of $\sigma$ such that  $H(\mathcal J_\sigma) = \phi\circ \tau(H_{\mathcal C})$ for some permutation $\phi$ on the vertex set of  $\tau(H_{\mathcal C})$, where the $=$ is in the sense of equal for labeled graphs. We claim that  the number of such $\sigma$'s is bounded by
		\begin{equation}\label{eq:number-of-configurations}
			\left(n-\sum_{i\in [l]}x_ir_i-\sum_{j\in [m]}y_js_j\right)!\times \prod_{i\in [l]}x_i!\operatorname{Aut}(C_i)^{x_i}\prod_{j\in [m]}y_j!\operatorname{Aut}(T_j)^{y_j}\,.
		\end{equation}
		Let $V_1$ be the vertex set of $\tau(H_\mathcal{C})$, then $|V_1|=\sum_{i\in [l]}x_ir_i+\sum_{j\in[m] }y_js_j$. It is clear that any aforementioned desired $\sigma$ can be decomposed into two permutations $\sigma_1$ and $\sigma_2$ on $V_1$ and $V\setminus V_1$, respectively. The number of choices for $\sigma_2$ is at most $(n-|V_1|)!$ (it may be strictly less than $(n-|V_1|)!$ since on $V\setminus V_1$ we are not allowed to produce another edge orbit that is entirely in $\mathcal H$). In order to bound the number of choices for $\sigma_1$, we use the following crucial observation: for any $u\in V_1$, we have $\sigma_1$ inhibits to an isomorphism between the two components of $\tau(H_\mathcal{C})$ which contain $u$ and $\sigma_1(u)$. That is to say, for any $u$ adjacent to $v$ in $\tau(H_\mathcal{C})$, $\sigma(u)$ is also adjacent to $\sigma(u)$  in $\tau(H_{\mathcal C})$;  similarly, for any $w$ adjacent to $z$ in $\tau(H_\mathcal{C})$, $\sigma^{-1}(w)$ is also adjacent to $\sigma^{-1}(z)$ in $\tau(H_\mathcal{C})$. This is true because of the definition of edge orbit and our requirement that $H(\mathcal J_\sigma)$ is entirely contained in $\mathcal H$.
		From this observation, for each $C_i$ (and similarly for $T_j$) we will ``permute'' its $x_i$ copies so that $\sigma$ maps one copy to its image under the permutation, and  within each copy of $C_i$ we have the freedom of choosing an arbitrary automorphism. In addition, the choice of such permutations and automorphisms completely determine $\sigma_1$. Therefore, the number of choices for $\sigma_1$ is bounded by $\prod_{i\in[l]}x_i!\operatorname{Aut}(C_i)^{x_i}\prod_{j\in[m]}y_j!\operatorname{Aut}(T_j)^{y_j}$. This proves \eqref{eq:number-of-configurations}. Combined with \eqref{eq:location-choice}, it yields \eqref{eq:number-of-sigma}.
	\end{proof}
	
	\begin{proof}[Proof of Proposition~\ref{prop:control-exp-moment}]
		We have 
		\begin{equation}\label{eq:est}
			\frac{1}{n!}\sum_\sigma \frac{p^{-|\mathcal{J}_\sigma|}}{\mathcal Q[\mathcal G\mid \pi^*,\mathcal J_\sigma]}\le\frac{1}{n!}\sum\frac{p^{-|E(H_\mathcal{C})|}}{\mathcal Q[\mathcal G\mid \pi^*,H(\mathcal{J}_\sigma)\cong H_\mathcal{C}]}\times|\sigma:H(\mathcal{J}_\sigma)\cong H_\mathcal{C}|\,,
		\end{equation}
		where the sum is over all possible representatives $H_\mathcal{C}$.
		We will bound $\mathcal Q[\mathcal G\mid \pi^*,\mathcal J_\sigma]$  by \eqref{eq:probability-of-admissible-condition-on-subgraph-2}. For each tree component with $s$ vertices it is clear that the number of edges is $s-1$, and crucially for each non-tree component with $r$ vertices, we use \eqref{eq:(i)-in-admissibility} to bound the number of edges in the component. In addition, we use \eqref{eq:number-of-sigma} to bound $|\sigma:H(\mathcal{J}_\sigma)\cong H_\mathcal{C}|$. Therefore, by enumerating all the possible isomorphic class representatives $H_\mathcal{C}$, we can upper-bound \eqref{eq:est} by
		\begin{equation*}
			\begin{aligned}\label{eq:bound-exp-moment-1}
				& \frac{1}{n!}\sum_{l,m\ge 0}\sum_{C_1,\dots,C_l\in \bigcup \mathfrak{C}_r}\sum_{T_1,\dots,T_m\in \bigcup \mathfrak T_s}\sum_{x_1,\dots,x_l>0}\sum_{y_1,\dots,y_m>0}[1+o(1)]\\
				\times&\ (cp)^{-\xi \sum_{i\in [l]}x_i|C_i|-\sum_{j\in [m]}y_j(|T_j|-1)}\left(n-\sum_{i\in [l]}x_i|C_i|-\sum_{j\in [m]}y_j|T_j|\right)!\\
				\times&\prod_{i\in [l]}\left(\operatorname{Aut}(C_i) t(C_i,\mathcal{H})\right)^{x_i}\prod_{j\in [m]}\left(\operatorname{Aut}(T_j) t(T_j,\mathcal H)\right)^{y_j}\,,
			\end{aligned}
		\end{equation*}
		where $c=c(\lambda,\delta)$ is the constant in \eqref{eq:probability-of-admissible-condition-on-subgraph-2}. By Stirling's formula we see 
		\begin{equation*}
			\frac{(n-k)!}{n!}\le \left(\frac{2e}{n}\right)^{k},\mbox{ for all } 0\le k\le n\,.
		\end{equation*}
		Write $c'={c}/{2e}$, then \eqref{eq:bound-exp-moment-1} can be further bounded by $1+o(1)$ multiples
		\begin{align}\label{eq:bound-exp-moment-2}
			&\sum_{l\ge 0}\sum_{C_1,\dots,C_l\in \bigcup\mathfrak C_r}\sum_{x_1,\dots,x_l>0}\left(\frac{1}{c'np^{\xi}}\right)^{x_1|C_1|+\dots+x_l|C_l|}\prod_{i\in [l]}\left(\operatorname{Aut}(C_i)t(C_i,\mathcal{H})\right)^{x_i} \nonumber\\
			\times&
			\sum_{m\ge 0}\sum_{T_1,\dots,T_m\in \bigcup\mathfrak T_s}\sum_{y_1,\dots,y_m>0}\left(\frac{1}{c'np}\right)^{y_1|T_1|+\dots+y_m|T_m|}\prod_{j\in[m]}\left(p\operatorname{Aut}(T_j)t(T_j,\mathcal{H})\right)^{y_j} \nonumber\\
			\le &\prod_{C\in\bigcup \mathfrak C_r}\left(1+\sum_{x>0}\left(\frac{\operatorname{Aut}(C)t(C,\mathcal H)}{(c'np^\xi)^{|C|}}\right)^x\right)\prod_{T\in \bigcup\mathfrak T_s}\left(1+\sum_{y>0}\left(\frac{p\operatorname{Aut}(T) t(T,\mathcal H)}{(c'np)^{|T|}}\right)^y \right)\,.
		\end{align}
		Under the assumption that $\mathcal H$ is admissible and the condition $np\ge np^\xi\ge n^{\delta_0}$ by \eqref{eq-xi-delta-assumption}, we get from \eqref{eq:count-of-tree} that
		\begin{equation}\label{eq:<1/2-2}
			\frac{p\operatorname{Aut}(T)t(T,\mathcal H)}{(c'np)^{|T|}}\le \frac{np(4\log n)^{2(|T|-1)}}{c'np\cdot(c'n^{\delta_0})^{|T|-1}}=o(1),\quad\mbox{ for all }T\in \mathfrak T_s\, .
		\end{equation}
		Similarly, we get from \eqref{eq:count-of-non-tree} that (recall $\delta<\delta_0$ by choice)
		\begin{equation}\label{eq:<1/2-1}
			\frac{\operatorname{Aut}(C)t(C,\mathcal H)}{(c'np^\xi)^{|C|}}\le \frac{|C|^3(2^{\xi+1}n^{\delta})^{|C|}}{(c'n^{\delta_0})^{|C|}}=o(1),  \quad\mbox{ for all } C\in \mathfrak{C}_r\, .
		\end{equation}
		Since  $\log(1+x)\le x$ for all $x>0$, we get from  \eqref{eq:<1/2-2} and \eqref{eq:<1/2-1} that the logarithm of \eqref{eq:bound-exp-moment-2} is bounded by $o(1)$ plus
		\begin{equation}\label{eq:logarithm-bound}
			[1+o(1)]\left[\sum_{C\in\bigcup \mathfrak C_r}\frac{\operatorname{Aut}(C)t(C,\mathcal H)}{(c'np^\xi)^{|C|}}+\sum_{T\in \bigcup\mathfrak T_s}\frac{p\operatorname{Aut}(T)t(T,\mathcal H)}{(c'np)^{|T|}}\right]\,.
		\end{equation}
		In order to bound \eqref{eq:logarithm-bound}, we first sum over all $C\in \mathfrak{C}_r$ and $T\in \mathfrak{T}_s$ then sum over $r, s$.  Applying this procedure and using \eqref{eq:count-of-tree} and \eqref{eq:count-of-non-tree}, we get that \eqref{eq:logarithm-bound} is at most 
		\begin{equation*}
			[1+o(1)] \left[\sum_{k\ge 2}\frac{k^3(2^{\xi+1}n^\delta)^k}{(c'n^{\delta_0})^k}+\sum_{k\ge 2}\frac{np(4\log n)^{k-1}}{(c'np)^k}\right]=o(1)\,,
		\end{equation*}
		which shows that the logarithm of \eqref{eq:bound-exp-moment-2} is $o(1)$. This implies that \eqref{eq:bound-exp-moment-1} is bounded by $1+o(1)$.  Combined with \eqref{eq:est}, this completes the proof of Proposition~\ref{prop:control-exp-moment}.
	\end{proof}

\end{document}